\documentclass[a4paper,12pt]{amsart}
\usepackage{amsmath,amsthm}
\usepackage{amsfonts,amsmath,amsthm, amssymb, mathrsfs}
\usepackage{float}
\usepackage{euscript,eufrak,verbatim}
\usepackage{graphicx}
\usepackage[usenames]{color}
\usepackage[colorlinks,linkcolor=red,anchorcolor=blue,citecolor=blue]{hyperref}
\usepackage{amsmath}
\usepackage{amsthm}
\usepackage[all]{xy}
\usepackage{graphicx} 
\usepackage{amssymb} 
\usepackage{enumerate}

\newtheorem{thm}{Theorem}[section]
\newtheorem{prop}[thm]{Proposition}
\newtheorem{df}[thm]{Defintion}

\newtheorem{cor}[thm]{Corollary}

\newtheorem*{theorem2}{Theorem A}

\newtheorem{rem}[thm]{Remark}

\def\Z{\mathbb{Z}}
\def\R{\mathbb{R}}
\def\d{d}
\def\diam{\text{\rm diam}}

\def\over{{\rm {\rm \overline{{mdim}}}}_M(T,f,d)}

\def\s{S_nf(x)}
\def\f{f\in C(X,\mathbb{R})}
\def\logf{\log\frac{1}{\epsilon}}

\makeatletter 
\@addtoreset{equation}{section}
\numberwithin{equation}{section}

\title{on variational principle  for upper metric mean dimension with potential}

\author{Rui Yang, Ercai Chen and Xiaoyao Zhou*
}
\address
{1.School of Mathematical Sciences and Institute of Mathematics, Nanjing Normal University, Nanjing 210023, Jiangsu, P.R.China}
\email{zkyangrui2015@163.com}
\email{ecchen@njnu.edu.cn}
\email{zhouxiaoyaodeyouxian@126.com}
\date{}

\begin{document}

\maketitle
\renewcommand{\thefootnote}{}
\footnote{*corresponding author}

\begin{abstract}
Borrowing  the  idea of  topological pressure determining measure-theoretical entropy in topological dynamical systems, we establish a variational principle for upper metric mean dimension with potential in terms of upper measure-theoretical metric mean dimension of invariant measures. Moreover, the notion of  equilibrium state  is introduced to characterize these measures  that attain the supremum of the variational principle.

\end{abstract}

\section{Introduction}

By a pair $(X,T)$ we mean a topological dynamical system (TDS for short), where $X$ is a  compact metrizable  topological space  and $T$ is a continuous self-map from $X$ to $X$.   We denote by $M(X),M(X,T)$  the sets of all Borel probability measures on $X$, all $T$-invariant Borel probability measures on $X$,  respectively. The continuous function space $C(X,\R)$ consisting of  all continuous real-valued functions on $X$ is  a normed linear space equipped with the supremum norm.

Mean topological dimension, introduced  by Gromov \cite{gromov},  is a  new  topological  invariant  for  topological dynamical systems.  Later,  Lindenstrauss and Weiss \cite{lw00}  introduced the concept of metric mean dimension  to capture the complexity of infinite topological entropy systems,  and  proved  that  metric mean dimension  is an upper bound of mean topological dimension. Since then,   pursuing the relation between  the  mean dimension theory   and other field  of dynamical systems  within the framework of infinite entropy theory  attracted abundant special attentions.     
It is  well-known that  the  classical variational principle \cite{rue73, wal75} states that topological pressure of a continuous potential  equals the supremum of measure-theoretical entropy plus  the integral  of the continuous potential taken over all invariant measures.  Readers can turn to  \cite{hy07,cfh08,lcc12,chz13,chu13,lmw18,bh20,hwz21} for  more  results  of this aspect, which involves  the variational principles of  topological pressure in the context of  different  dynamical systems.  This naturally leads to a question how to  establish  the proper variational principles for infinite entropy systems. Very recently,  Lindenstrauss  and Tsukamoto's pioneering work \cite{lt18} gave the first  analogue of classical variational  principle for  metric mean dimension. 
\begin{theorem2}
Let $(X,T)$ be a TDS with a metric $\d$. Then
\begin{align*}
{ \rm \overline{mdim}}_M(T,X,d)&=\limsup_{\epsilon \to 0}\frac{1}{\logf}\sup_{\mu \in M(X,T)}R_{\mu, L^{\infty}}(\epsilon),\\
{ \rm {\rm \underline{mdim}}}_M(T,X,d)&=\liminf_{\epsilon \to 0}\frac{1}{\logf}\sup_{\mu \in M(X,T)}R_{\mu, L^{\infty}}(\epsilon),
\end{align*}
  where ${\rm \overline{mdim}}_M(T,X,d), {\rm {\rm \underline{mdim}}}_M(T,X,d)$ are upper and lower metric mean dimensions of $X$, respectively, and  $R_{\mu, L^{\infty}}(\epsilon)$  is  $L^{\infty}$-rate distortion dimension of $\mu$.
 \end{theorem2}
Furthermore,  the  aforementioned  variational principles   hold  if we replace $R_{\mu, L^{\infty}}(\epsilon)$  by Katok  entropy  \cite{vv17},  R$\bar{e}$nyi information function \cite{gs20},  Brin-Katok local entropy, Shapira  entropy and local entropy function \cite{shi}. For systems with marker  property,  Lindenstrauss and Tsukamoto \cite{lt19}  established  double variational principles for mean dimension, which was extended  to    mean dimension with potential by Tsukamoto  \cite{t20}.  See also \cite{ccl22,w22} for the  variational principles of metric mean dimension with potential. We aim  to  inject  ergodic theoretic ideas  into mean dimension theory by  obtaining  a new variational principle for upper metric mean dimension with potential. Compared with the classical variational principle for topological pressure, a  satisfactory  Lindenstrauss-Tsukamoto's variational principle  is   exchanging the  order of $\limsup_{\epsilon \to 0}$ and $\sup_{\mu \in M(X,T)}$ for $L^{\infty}$. Unfortunately, Lindenstrauss  and Tsukamoto  \cite[Section VIII]{lt18} constructed  a  counter-example showing  the order of  the operations of  $\limsup$ (or $\liminf$) and $\sup$ in Theorem A can not be  exchanged. 
 
  In general, it  is not easy to define a  proper quantity,  which we call \emph{measure-theoretical metric mean dimension of invariant measures},  that  does not depend on the metric of the phase space. To  solve this  difficulty,  we borrow the idea of topological pressure determining measure-theoretical entropy presented in \cite{wal75}  to  define a new  measure-theoretical  metric mean dimension for Borel  probability measures, and  using  convex  approach link it with metric mean dimension with potential.
  \begin{thm}\label{thm 1.1}
  Let $(X,T)$ be a  dynamical system  with a metric $\d$ such that ${\rm \overline{mdim}}_M(T,X,d)<\infty$. Then for all $\f$,  
  $$\over=\max_{\mu \in M(X,T)} \{{F}(\mu,d)+\int fd\mu\},$$
 where $\over$ is upper metric mean dimension with potential $f$, and $F(\mu,d)$  is  the  upper measure-theoretical  metric mean dimension of $\mu$.  
  \end{thm}

The rest of this paper is organized as follows. In section 2,  we recall the definition of  metric mean dimension with potential and derive  some basic properties for it. In section 3, we prove Theorem 1.1. In section 4,  We introduce the notion of equilibrium  state for upper metric mean dimension  with potential.

\section{ metric mean dimension with potential}

In this section, we recall the precise definition  of  metric mean dimension  with potential and derive some basic properties for it.

Let $n\in \mathbb{N}$ and $f\in C(X,\mathbb{R})$. For $x\in X$,  we  define $S_nf(x)=\sum_{j=0}^{n-1}\limits f(T^jx)$. Given $n\in \mathbb{N}$, the $n$-th Bowen metric $d_n$  on $X$  is given  by $$d_n(x,y):=\max_{0\leq j\leq n-1}\limits d(T^{j}(x),T^j(y))$$ for any $x,y \in X$. Then the \emph{Bowen open  ball}  $B_n(x,\epsilon)$ of radius $\epsilon$ and  order $n$ in the metric $d_n$  is given by  $$B_n(x,\epsilon):=\{y\in X: d_n(x,y)<\epsilon\}.$$ 
Given $\d$ and $\f$,  set 
\begin{align*}
&R_n(X,f,d,\epsilon)\\
=&\inf
\left\{\sum_{i=1}^{n}(1/\epsilon)^{\sup_{x\in U_i}\s}:
\begin{array}{l} 
X=U_1\cup\cdots \cup U_n~\text{with}\\
\diam(U_i, d_n)<\epsilon ~ \text{for all}~ i=1,...,n.
\end{array}
\right\}
\end{align*}
 and  $$R(X,f,d,\epsilon)=\lim_{n\to\infty}\limits\frac{\log R_n(X,f,d,\epsilon)}{n}.$$
 The limit exists  since $a_n:=\log R_n(X,f,d,\epsilon)$  is subadditive in $n$.  The \emph{upper metric mean dimension  of $X$ with potential $f$} \cite{t20} is given  by 
$${\rm {\rm \overline{{mdim}}}}_M(T,X,f,d)=\limsup_{\epsilon \to 0}\frac{R(X,f,d,\epsilon)}{\logf}.$$

If the system $(X,T)$ is clear, we sometimes  write ${\rm {\rm \overline{{mdim}}}}_M(T,f,d)$ instead of ${\rm {\rm \overline{{mdim}}}}_M(T,X,f,d)$. Replacing  $\limsup_{\epsilon \to 0}\limits$ with $\liminf_{\epsilon \to 0}\limits$, we  denote by  ${\rm {\rm \underline{{mdim}}}}_M(T,X,f,d)$  the lower metric mean dimension with potential $f$. If ${\rm {\rm \overline{{mdim}}}}_M(T,X,f,d)={\rm {\rm \underline{{mdim}}}}_M(T,X,f,d)$,  we call the common value  \emph{metric mean dimension of $X$ with  potential $f$} and  denote it by ${\rm {mdim}}_M(T,X,f,d)$. When $f=0$ is a zero potential,   we write ${\rm {\rm \overline{{mdim}}}}_M(T,X,d)={\rm {\rm \overline{{mdim}}}}_M(T,X,0,d)$.  This  exactly recovers the definition of  metric mean dimension of $X$ introduced by Lindenstrauss and Weiss \cite{lw00}. Obviously, the (upper/ lower) metric mean dimension with potential  depends  on the  compatible metrics on $X$.

Analogous to the definitions of  classical topological pressure in thermodynamic formalism, the upper (or lower) metric mean dimension with potential can be  equivalently  defined by spanning sets and separated  sets \cite{b71,w82}.

 A set $E\subset X$ is  \emph{an $(n,\epsilon)$-\emph{spanning set} of $X$} if  for any $x \in X$, there  exists  $y\in E$ such that $d_n(x,y)<\epsilon.$  A set $F\subset X$ is \emph{an $(n,\epsilon)$-separated set of $X$} if $d_n(x,y)\geq\epsilon$ for any $x,y \in F$ with $x\not= y$. 
 
 Define
\begin{align*}
P_n(X,f,d,\epsilon)&=\sup\{\sum_{x\in F}(1/\epsilon)^{\s}: F~\mbox{is an}~ (n,\epsilon)\mbox{-separated set of}~X\},\\
Q_n(X,f,d,\epsilon)&=\inf\{\sum_{x\in E}(1/\epsilon)^{\s}: E~\mbox{is an}~ (n,\epsilon)\mbox{-spanning set of}~X\}.
\end{align*}

The following proposition  enables us to pay attention to  outer limit $\limsup_{\epsilon \to 0}$ rather than inner limits $\limsup_{n\to \infty}$ and $\liminf_{n\to \infty}$  when we  compute the upper (or lower) metric mean dimension with potential. 
\begin{prop}\label{prop 2.1}
Let $(X,T)$ be a TDS with a metric $\d$ and $\f$. Then
\begin{align}
{\rm {\rm \overline{{mdim}}}}_M(T,f,d)&=\limsup_{\epsilon \to 0}\limsup_{n\to \infty} \frac{\log Q_n(X,f,d,\epsilon)}{n\logf}\\
&=\limsup_{\epsilon \to 0}\liminf_{n\to \infty} \frac{\log Q_n(X,f,d,\epsilon)}{n\logf}\label{inequ 2.2}\\
&=\limsup_{\epsilon \to 0}\limsup_{n\to \infty} \frac{\log P_n(X,f,d,\epsilon)}{n\logf}\label{inequ 2.3}\\
&=\limsup_{\epsilon \to 0}\liminf_{n\to \infty} \frac{\log P_n(X,f,d,\epsilon)}{n\logf},
\end{align}
it is also valid for${\rm {\rm \underline{{mdim}}}}_M(T,f,d)$ by changing $\limsup_{\epsilon \to 0}\limits$  into $\liminf_{\epsilon \to 0}\limits$. 
\end{prop}

\begin{proof}
We first show $(2.1)=(2.2)$. Let $\epsilon >0$ and we set $\gamma(\epsilon):=\sup\{|f(x)-f(y)|:d(x,y)<\epsilon\}$, then $\gamma(\epsilon) \to 0$ as $\epsilon \to 0$. Let $0<\epsilon <1$.  Let $E$ be an $(n,\frac{\epsilon}{2})$-spanning set of $X$, then $X=\cup_{x\in E}B_n(x,\frac{\epsilon}{2})$. This implies that
\begin{align*}
R_n(X,f,d,\epsilon)&\leq\sum_{x\in E}(1/\epsilon)^{\sup_{y\in B_n(x,\frac{\epsilon}{2})}S_nf(y)}\\
&\leq \sum_{x\in E}(1/\epsilon)^{\s+n\gamma(\epsilon)}\\
&\leq  (1/\epsilon)^{n\gamma(\epsilon)}2^{n||f||}\sum_{x\in E}(2/\epsilon)^{\s}.
\end{align*}
We obtain that \begin{align}
\lim_{n\to\infty}\limits \frac{\log R_n(X,f,d,\epsilon)}{n}\leq \liminf_{n\to\infty}\limits \frac{\log Q_n(X,f,d,\frac{\epsilon}{2})}{n}+\gamma(\epsilon)\log\frac{1}{\epsilon}+||f||\log2.
\end{align}  

Let $X=U_1\cup \cdots \cup U_n$  be an open cover of $X$ with $ \diam(U_i, d_n)<\frac{\epsilon}{2}, i=1,...,n.$ Choose $x_i\in U_i$ for every $1\leq i\leq n$.  We have $\{x_1,..,x_n\}$  is an $(n,\epsilon)$-spanning set of $X$. Similarly, we can deduce that 
\begin{align}
\limsup_{n\to\infty}\limits \frac{\log Q_n(X,f,d,\epsilon)}{n}\leq \lim_{n\to\infty}\limits \frac{\log R_n(X,f,d,\frac{\epsilon}{2})}{n}+\gamma(\epsilon)\log\frac{1}{\epsilon}+||f||\log2.
\end{align}  
Combing the facts (2.5) and (2.6), this shows $(2.1)=(2.2)$.

It remains to show $(2.1)=(2.3)$ and $(2.2)=(2.4)$. Here, we only need to show $(2.1)=(2.3)$  since $(2.2)=(2.4)$ can be proved  in a similar manner. Note that an  $(n,\epsilon)$-separated set  with the maximal cardinality of $X$ is also an  $(n,\epsilon)$-spanning set of $X$, this yields that $Q_n(X,f,d,\epsilon) \leq P_n(X,f,d,\epsilon)$.  On the  other hand,  let $E$ be an  $(n,\frac{\epsilon}{2})$-spanning set of $X$ and $F$ be an $(n,\epsilon)$-separated set of $X$.  We define a map $\Phi: F \longrightarrow E$ by assigning each $x \in F$ to $\Phi(x)\in E$ with $d_n(x,\Phi(x))< \frac{\epsilon}{2}$. It is easy  to show  the map is injective.
\begin{align*}
 & \sum_{x\in E}(2/\epsilon)^{\s}\\
\geq& \sum_{x\in F}(1/\epsilon)^{S_nf(\Phi(x))}2^{-n|||f||}\\ 
\geq&  \sum_{x\in F}(1/\epsilon)^{S_nf(x)-n\gamma(\epsilon)}2^{-n|||f||}  .
\end{align*}
This shows
 \begin{align}
\limsup_{n\to\infty}\limits \frac{\log Q_n(X,f,d,\frac{\epsilon}{2})}{n}\geq \limsup_{n\to\infty}\limits \frac{\log P_n(X,f,d,\epsilon)}{n}-\gamma(\epsilon)\log\frac{1}{\epsilon}-||f||\log2.
\end{align}  
This completes the proof.
\end{proof}

The next propositions presents  the basic properties  of  the upper (lower) metric mean dimension with potential.
\begin{prop}\label{prop 2.2}
Let $(X,T)$ be a TDS with a metric $\d$ and $f,g\in C(X,\mathbb{R})$. Then
\begin{enumerate}

\item If  $f\leq g$, then  ${\rm {\rm \overline{{mdim}}}}_M(T,f,d)\leq {\rm \overline{{mdim}}}_M(T,g,d)$.\\
\item  $ {\rm \overline{{mdim}}}_M(T,f+c,d)={\rm {\rm \overline{{mdim}}}}_M(T,f,d)+c$ for  any $c\in \R$.
\item  $ {\rm \overline{{mdim}}}_M(T,X,d)+\inf f \leq {\rm {\rm \overline{{mdim}}}}_M(T,f,d) \leq {\rm \overline{{mdim}}}_M(T,X,d)+\sup f$.
\item  ${\rm \overline{{mdim}}}_M(T,\cdot,d): C(X,\R)\longrightarrow \R\cup\{\infty\}$ is either finite value or constantly $\infty$.
\item If ${\rm {\rm \overline{{mdim}}}}_M(T,f,d)\in \R$ for all $\f$, then 
$$|{\rm {\rm \overline{{mdim}}}}_M(T,f,d)-{\rm \overline{{mdim}}}_M(T,g,d)|\leq ||f-g||$$
and ${\rm \overline{{mdim}}}_M(T,\cdot,d)$ is convex on $C(X,\R)$. 

\item ${\rm \overline{{mdim}}}_M(T,f+g,d)\leq {\rm {\rm \overline{{mdim}}}}_M(T,f,d)+{\rm \overline{{mdim}}}_M(T,g,d)$.
\item  If  $c\geq 1$, then ${\rm \overline{{mdim}}}_M(T,cf,d)\leq c{\rm {\rm \overline{{mdim}}}}_M(T,f,d)$; if  $c\leq 1$, then ${\rm \overline{{mdim}}}_M(T,cf,d)\geq c{\rm {\rm \overline{{mdim}}}}_M(T,f,d)$.
\item $|{\rm {\rm \overline{{mdim}}}}_M(T,f,d)|\leq {\rm \overline{{mdim}}}_M(T,|f|,d)$.
\end{enumerate}
\end{prop}

\begin{proof}
We give the proof of Proposition 2.2 one by one.

(1-3) can be proved by using the definition of $\over$.

(4)Note that ${\rm \overline{{mdim}}}_M(T,X,d)\in[0,\infty)\cup\{\infty\}$, and ${\rm \overline{{mdim}}}_M(T,X,d)<\infty$ if and only if $\over\in \R$ for all  $\f$ by (3), this shows (4).

(5)  Let $0<\epsilon <1$ and $F$ be an  $(n,\epsilon)$-separated set of $X$. Then we have
$$\sum_{x\in F}(1/\epsilon)^{\s}\leq \sum_{x\in F}(1/\epsilon)^{S_ng(x)+n||f-g||},$$
which implies that  ${\rm {\rm \overline{{mdim}}}}_M(T,f,d)\leq {\rm \overline{{mdim}}}_M(T,g,d)+||f-g||.$
Exchanging the role of $f$ and $g$, we obtain that  $${\rm \overline{{mdim}}}_M(T,g,d)\leq {\rm {\rm \overline{{mdim}}}}_M(T,f,d)+||f-g||.$$
Let $p\in[0,1]$ and $f,g\in C(X,\R)$, and let $0<\epsilon<1$. Let $F$ be an  $(n,\epsilon)$-separated set of $X$. Using H$\ddot{o}$lder's inequality, we have
\begin{align*}
\sum_{x\in F}(1/\epsilon)^{pS_nf(x)+(1-p)S_ng(x)}\leq (\sum_{x\in F}(1/\epsilon)^{S_nf(x)})^p(\sum_{x\in F}(1/\epsilon)^{S_ng(x)})^{(1-p)},
\end{align*}
 which yields that  ${\rm \overline{{mdim}}}_M(T,pf+(1-p)g,d)\leq p{\rm {\rm \overline{{mdim}}}}_M(T,f,d)+(1-p){\rm \overline{{mdim}}}_M(T,g,d)$.
 
 (6) Let $0<\epsilon <1$ and $F$ be an  $(n,\epsilon)$-separated set of $X$. Then we have
 $$\sum_{x\in F}(1/\epsilon)^{S_n(f+g)(x)}\leq \sum_{x\in F}(1/\epsilon)^{S_nf(x)}\cdot \sum_{x\in F}(1/\epsilon)^{S_ng(x)},$$
 which  implies the desired result.
 
 (7) If $a_1,...,a_k$ are $k$ positive real numbers with $\sum_{i=1}^ka_i=1$, then $\sum_{i=1}^ka_i^c\leq1$ if $c\geq1$;$\sum_{i=1}^ka_i^c\geq1$ if $c\leq1$;
 
 Let $0<\epsilon <1$ and $F$ be an  $(n,\epsilon)$-separated set of $X$. Then we have
  $$\sum_{x\in F}(1/\epsilon)^{cS_nf(x)}\leq (\sum_{x\in F}(1/\epsilon)^{S_nf(x)})^c$$ if $c\geq 1$;
  
$$\sum_{x\in F}(1/\epsilon)^{cS_nf(x)}\geq (\sum_{x\in F}(1/\epsilon)^{S_nf(x)})^c$$ if $c\leq1$;

We get the desired result  after taking the corresponding limits.

(8)  Since $-|f|\leq f\leq |f|$, then ${\rm \overline{{mdim}}}_M(T,-|f|,d)\leq {\rm {\rm \overline{{mdim}}}}_M(T,f,d)\leq {\rm \overline{{mdim}}}_M(T,|f|,d)$ by (1). By (7), we have  ${\rm \overline{{mdim}}}_M(T,-|f|,d)\geq -{\rm \overline{{mdim}}}_M\\(T,|f|,d)$. This shows (8).
  
\end{proof}

The following proposition establishes the product formula for  metric mean dimension formula with potential.
\begin{prop}
Let $(X_i,T_i)$ be a TDS with a metric $d_i\in \mathscr{D}(X_i)$ and $f_i\in C(X_i,\mathbb{R})$,  $i=1,2$, and $(X_1\times X_2, T_1\times T_2)$ is a product system defined by $T_1\times T_2:(x_1,x_2)\longmapsto(T_1x_1,T_2x_2)$. Then 
\begin{align}
{\rm \overline{{mdim}}}_M(T_1\times T_2,X_1\times X_2,f,d)\leq {\rm \overline{{mdim}}}_M(T_1,X_1,f_1,d_1)+ {\rm \overline{{mdim}}}_M(T_2,X_2,f_2,d_2)\\
{\rm \underline{{mdim}}}_M(T_1\times T_2,X_1\times X_2,f,d)\geq {\rm \underline{{mdim}}}_M(T_1,X_1,f_1,d_1)+ {\rm \underline{{mdim}}}_M(T_2,X_2,f_2,d_2).
\end{align}
where 
$d((x_1,x_2),(y_1,y_2))=\max\{d_1(x_1,y_1),d_2(x_2,y_2)\}$ for any $(x_1,x_2),\\(y_1,y_2)\in X_1\times X_2$,  and $f:X_1\times X_2 \longrightarrow \R$ is defined by $f(x_1,x_2)=f_1(x_1)+f_2(x_2).$
\end{prop}

\begin{proof}
Let $0<\epsilon <1$. Let  $E$ be an  $(n,\epsilon)$-spanning  set of $X_1$ and  $F$ be an  $(n,\epsilon)$-spanning  set of $X_2$. Then $E\times F$ is an $(n,\epsilon)$-spanning  set of $X_1 \times X_2$. This gives us that

\begin{align*}
Q_n(X_1\times X_2,f,d,\epsilon)&\leq\sum_{(x,y)\in E\times F}(1/\epsilon)^{S_nf(x,y)}\\
&=\sum_{x\in E}(1/\epsilon)^{S_nf_1(x)}\cdot \sum_{y\in F}(1/\epsilon)^{S_nf_2(y)},
\end{align*}
which implies that $$Q_n(X_1\times X_2,f,d,\epsilon)\leq Q_n(X_1,f_1,d_1,\epsilon)Q_n( X_2, f_2,d_2,\epsilon).$$

We  get  inequality (2.8) by (\ref{inequ 2.3}). Inequality (2.9) can be  obtained by using the definition of separated set of  lower metric mean dimension with potential  and Proposition \ref{prop 2.1}.
\end{proof}

A power rule of  metric mean dimension  with potential  is stated as follows.

\begin{prop}
Let $(X,T)$ be a TDS with a metric $d$ and $f\in C(X,\mathbb{R})$. Then  for any  positive number $k\geq2$, we have
\begin{align}
{\rm \overline{{mdim}}}_M(T^k,S_kf,d)\leq k{\rm {\rm \overline{{mdim}}}}_M(T,f,d),
\end{align}
where $S_kf(x)=\sum_{j=0}^{k-1}f(T^jx)$.

Additionally,  if  $T$ is  a  Lipschitz mapping on $X$,  then 
$${\rm \overline{{mdim}}}_M(T^k,S_kf,d)=k{\rm {\rm \overline{{mdim}}}}_M(T,f,d).$$
\end{prop}

\begin{proof}
Let $0<\epsilon <1$. Let  $E$ be an  $(nk,\epsilon)$-spanning  set of $X$  for $T$. Then  $E$ is an  $(n,\epsilon)$-spanning  set of $X$  for $T^k$.  It  follows that 
\begin{align*}
Q_n(X,S_kf,d,\epsilon)&\leq\sum_{x\in E}(1/\epsilon)^{S_n(S_kf(x))}\\
&=\sum_{x\in E}(1/\epsilon)^{S_{nk}f(x)},
\end{align*}
which implies  that $Q_n(X,S_kf,d,\epsilon)\leq Q_{nk}(X,f,d,\epsilon).$
Hence,  we have 
\begin{align*}
\limsup_{n\to \infty}\frac{\log Q_n(X,S_kf,d,\epsilon)}{n}&\leq  k\cdot\limsup_{n\to \infty}\frac{\log Q_{nk}(X,f,d,\epsilon)}{nk}\\
&\leq k\cdot\limsup_{n\to \infty}\frac{\log Q_{n}(X,f,d,\epsilon)}{n},
\end{align*}
which implies inequality (2.10).

Now assume that $L>0$ is a constant satisfying $d(Tx,Ty)\leq L\cdot d(x,y)$ for any $x,y \in X$. Fix $\alpha >0$.  Let $0<\epsilon <1$  and $C=(max\{L,1\}+\alpha)^k>1$. Then $d_k(x,y)\leq C\epsilon$ if $d(x,y)\leq \epsilon$ for any $x,y\in X$. This shows if  $E$ is  an  $(n,\epsilon)$-spanning  set of $X$  for $T^k$, then  $E$ is an  $(nk,C\epsilon)$-spanning  set of $X$  for $T$.
Similarly, we have
\begin{align*}
\limsup_{n\to \infty}\frac{\log Q_n(X,S_kf,d,\epsilon)}{n}&\geq k\cdot\liminf_{n\to \infty}\frac{\log Q_{nk}(X,f,d,C\epsilon)}{nk}\\
&\geq k\cdot\liminf_{n\to \infty}\frac{\log Q_{n}(X,f,d,C\epsilon)}{n}.
\end{align*}
By  (\ref{inequ 2.2}) and using the fact $\lim_{\epsilon \to 0}\limits\frac{\log \frac{1}{C\epsilon}}{\logf}=1$, we  have 
${\rm \overline{{mdim}}}_M(T^k,S_kf,d)\geq k{\rm {\rm \overline{{mdim}}}}_M(T,f,d).$ This completes the proof.

\end{proof}

\begin{rem}
By virtue  of the continuity of $T$, for any $\epsilon >0$  there exists $\delta(\epsilon)>0$ such that $d_k(x,y)<\epsilon$  for any  $x,y \in X$ with $d(x,y)<\delta$.  One can similarly deduce that
\begin{align*}
&\limsup_{\epsilon \to 0}\frac{\log\frac{1}{\delta (\epsilon)}}{\logf}\cdot\limsup_{n\to \infty}\frac{\log Q_n(X,S_kf,d,\delta)}{n\log \frac{1}{\delta(\epsilon)}}\\
\geq& k\cdot\limsup_{\epsilon \to 0}\liminf_{n\to \infty}\frac{\log Q_{n}(X,f,d,\epsilon)}{n\logf}.
\end{align*}
However, it is hard to determine  whether $\limsup_{\epsilon \to 0}\frac{\log\frac{1}{\delta (\epsilon)}}{\logf}$ is equal to $1$without the assumption of Lipschitz condition on the continuous potential.  
\end{rem}

\section{Proof of Theorem 1.1}
In this section, we borrow the idea of topological pressure determining measure-theoretical  entropy to define the upper measure-theoretical  metric mean dimension for  probability measures,  and further  prove Theorem 1.1.

The classical variational principle states that  for any $\f$, one has 
$$P(T,f)=\sup_{\mu \in M(X,T)}\{h_{\mu}(T)+\int f d\mu\},$$
where $P(T,f)$  denotes the  topological pressure of $f$ w.r.t. $T$, and $h_{\mu}(T)$  is  the  measure-theoretical  entropy of $\mu$ w.r.t. $T$. Suppose that  topological entropy  $h_{top}(T):=P(T,0)<\infty$. Using the variational principle for topological pressure,  Walters \cite[Theorem  9.12]{w82} proved that for each fixed $\mu_0\in M(X,T)$ the entropy map  of $T$ is upper semi-continuous at $\mu_0$ if and only if
\begin{align*}
h_{\mu_0}(T)&=\inf\{P(T,f)-\int fd\mu_0:f\in C(X,\R)\}\\
&=\inf_{f\in \hat{\mathcal{A}}}\int fd\mu_0,
\end{align*}
where $\hat{\mathcal{A}}=\{\f:P(T,-f)=0\}$.
In other words,  for such systems measure-theoretical  entropy is completely determined by  topological pressure.  We  follow this idea to define the so-called upper measure-theoretical metric mean dimension $F(\mu,d)$  for each probability measure on $X$.

\begin{df}\label{def 3.4}
Let  $(X,T)$ be a dynamical system with a metric $\d$ such that ${\rm \overline{{mdim}}}_M(T,X,d)<\infty$, and let  $\mu \in M(X)$. The  upper measure-theoretical  metric mean dimension of $\mu$ is defined as 
 $$F(\mu,d)=\inf_{f\in \mathcal{A}}\int fd\mu,$$ where $\mathcal{A}=\{\f:{\rm \overline{{mdim}}}_M(T,-f,d)=0\}$.
\end{df}

\begin{rem}\label{rem 3.5}
(1) We emphasize "metric" in definition 3.1 because $F(\mu,d)$ depends on the  choice of the metric $\d$. 

(2) If ${\rm \overline{{mdim}}}_M(T,X,d)=\infty$, we have $\over=\infty$ for all $C(X,\R)$ by Proposition 2.2,(3). In  this case,  $\mathcal{A} $ is  an  empty  set. Thus, we set $F(\mu,d)=\inf \emptyset=\infty$ for all $\mu \in M(X)$.



\end{rem}

\begin{prop} \label{prop 3.6}
Let  $(X,T)$ be a TDS with a metric $\d$ such that ${\rm \overline{{mdim}}}_M(T,X,d)<\infty$. Then $F(\mu,d)$ is  an upper  semi-continuous, concave  function on  $M(X)$.
\end{prop}



\begin{prop}\label{prop 3.4}
Let $(X,T)$ be a TDS with a metric $\d$ and $f$. Then for all $g\in C(X,\mathbb{R})$,
$${\rm \overline{mdim}}_M(T,f,d)
={\rm \overline{mdim}}_M(f+g\circ T-g)$$
\end{prop}

\begin{proof}
It can be proved using the fact if  $F$ is an  $(n,\epsilon)$-separated set of $X$, then 
$$\sum_{x\in F}(1/\epsilon)^{S_nf(x)-2||f-g||}\leq \sum_{x\in F}(1/\epsilon)^{S_n(f+g\circ T-g)}\leq \sum_{x\in F}(1/\epsilon)^{S_nf(x)+2||f-g||}.$$
\end{proof}

We invoke the  well-known  separation theorem of convex sets \cite[p. 417]{ds88}  for  the proof,  which states that if  $K_1, K_2$ are disjoint closed  subsets of a local convex linear topological space $V$,  and  $K_1$ is compact, then  there exists a continuous real-valued  linear  functional $\varphi$ on $V$ such that $\varphi(x)<\varphi(y)$ for any $x\in K_1$ and $y\in K_2$.

\begin{proof}[Proof of Theorem \ref{thm 1.1}]
We divide the proof into two steps.

Step 1. we  show 
$$ \over=\max_{\mu \in M(X)}\{F(\mu,d)+\int fd\mu\}.$$

 As $\over-f\in\mathcal{A}$,  we have $$F(\mu,d)\leq \int\over-fd\mu,$$
 which  shows for all  $\mu \in M(X)$,  ${F(\mu,d)+\int fd\mu} \leq \over$.
 
So the remaining is to show $$\over\leq\sup_{\mu \in M(X)} \{{F}(\mu, d)+\int fd\mu\}.$$ 
It is proved  by showing for any $\epsilon >0$, there exists $\mu\in M(X)$ such that  $\over-\epsilon < {{F}(\mu, d)+\int fd\mu}$, or equivalently $F(\mu,d)+\int f_1d\mu+\epsilon >0$, where $f_1=f-\over$.

Let $\mathcal{C}:=\{g\in C(X,\R):{\rm \overline{{mdim}}}_M(T,-g,d)\leq 0\}$. Then  $\mathcal{C}$  is a closed convex subset of $C(X,\R)$, where the convexity  of $\mathcal{C}$ follows from Proposition \ref{prop 2.2}, (5). Recall  $\mathcal{A}=\{f\in C(X,\R):{\rm \overline{{mdim}}}_M(T,-f,d)= 0\}$. Then $\mathcal{A}\subset \mathcal{C}$. Since ${\rm \overline{{mdim}}}_M(T,f_1,d)=0$, one has $-(f_1+\frac{\epsilon}{2})\notin \mathcal{C}$ and hence $-f_1\notin \mathcal{C}+\frac{\epsilon}{2}$. Let $K_1=\{-f_1\}$ and $K_2=\mathcal{C}+\frac{\epsilon}{2}$. Applying  the  separation theorem of convex sets to  the sets $K_1, K_2$,there exists a  continuous real-valued linear  functional  $L$ on $C(X,\R)$ such that $L(-f_1)<L(g)$ for  all $g\in \mathcal{C}+\frac{\epsilon}{2}.$ This yields that $\inf_{g\in \mathcal{C}+\frac{\epsilon}{2}}L(g)+L(f_1)\geq 0$.

Next we show  that $L$ is a positive linear functional  on $C(X,\R)$ and $L(1)>0$.
Let $g\in C(X,\R)$ with $g\geq 0$. For any $c>0$, one has
\begin{align*}
&{\rm \overline{{mdim}}}_M(T,-(cg+1+{\rm \overline{{mdim}}}_M(T,X,d)),d)\\
= &{\rm \overline{{mdim}}}_M(T,-cg,d)-1-{\rm \overline{{mdim}}}_M(T,X,d)\\
\leq &{\rm \overline{{mdim}}}_M(T,0,d)-1-{\rm \overline{{mdim}}}_M(T,X,d)=-1<0, 
\end{align*}
 which  implies that  $cg+1+{\rm \overline{{mdim}}}_M(T,X,d)\in \mathcal{C}$. Thus, $L(-f_1)<cL(g)+L(1+{\rm \overline{{mdim}}}_M(T,X,d))$. This shows  $L(g)\geq 0$; otherwise we have $L(-f_1)=-\infty$ by letting $c\to \infty$. Since  $L$ is not  constantly  equal to zero,  we can choose $g\in C(X,\R)$ such that $L(g)>0$ with the supremum norm $||g||<1$. Then $L(1)=L(g)+L(1-g)>0$.
 
 By  Riesz representation theorem,  there exists a $\mu \in M(X)$ such that
 $$\frac{L(g)}{L(1)}=\int gd\mu$$
 for all $g\in C(X,\R).$ Since $\mathcal{A}+\frac{\epsilon}{2}\subset  \mathcal{C}+\frac{\epsilon}{2}$, one has
\begin{align*}
 F(\mu,d)+\epsilon&=\inf_{f\in \mathcal{A}}\int fd\mu+\frac{\epsilon}{2}+\frac{\epsilon}{2}\\
 &=\inf_{f\in\mathcal{A}+\frac{\epsilon}{2}}\int fd\mu+\frac{\epsilon}{2}\geq \inf_{f\in\mathcal{C}+\frac{\epsilon}{2}}\int fd\mu+\frac{\epsilon}{2}.
\end{align*}
Therefore, \begin{align*}
F(\mu,d)+\int f_1d\mu+\epsilon &\geq \frac{L(f_1)}{L(1)}+\inf_{f\in\mathcal{C}+\frac{\epsilon}{2}}\int fd\mu+\frac{\epsilon}{2}\\
& = \frac{L(f_1)}{L(1)}+\inf_{g\in \mathcal{C}+\frac{\epsilon}{2}}\frac{L(g)}{L(1)}+\frac{\epsilon}{2}\\
&= \frac{1}{L(1)}(L(f_1)+\inf_{g\in \mathcal{C}+\frac{\epsilon}{2}}{L(g)}) +\frac{\epsilon}{2}>0.
\end{align*}
The fact that $F(\mu,d)$ is upper semi-continuous on $M(X)$ guarantees  the supremum can be attained for some measure $\mu \in M(X)$.

Step 2. we  show 
$$ \over=\max_{\mu \in M(X,T)}\{{F}(\mu, d)+\int fd\mu\}.$$
By step 1,  we can choose  $\mu_0 \in M(X)$ such that  $\over=F(\mu_0,d)+\int fd\mu_0$. By Proposition
\ref{prop 3.4} and Step l, for all $g\in C(X,R)$ one has
\begin{align*}
F(\mu_0,d)+\int fd\mu_0&={\rm \overline{mdim}}_M(T,f,d)\\
&={\rm \overline{mdim}}_M(f+g\circ T-g)\\
&\geq  F(\mu_0,d)+\int f+g\circ T-gd\mu_0.
\end{align*} 
This shows  $\int gd\mu_0 \geq  \int g\circ Td\mu_0$. Thus,  one has $\mu_o \in M(X,T)$ and  $$\over= F(\mu_0,d)+\int fd\mu_0\leq  \sup_{\mu \in M(X,T)}\{F(\mu,d)+\int fd\mu\}.$$
By Step 1, this completes the proof.

\end{proof}

\begin{rem}
Following the line of the proof of Theorem 1.1, we can not obtain an analogous result for lower metric mean dimension with potential  due to the lack of convexity on $C(X,\R)$.
\end{rem}

\section{Equilibrium states }
In this section, we  introduce  the notion of equilibrium states for upper metric mean dimension with potential to characterize the  members of $M(X)$ that attain the supremum.

The following notion is  an  analogue of the concept of equilibrium state for topological pressure in  thermodynamic formalism  \cite{w82} (or known as maximal entropy measure for topological  entropy).

\begin{df}
Let $(X,T)$ be a TDS with a metric $\d$ staisfying ${\rm \overline{mdim}}_M(T,X,d)<\infty$, and let $\f$.  A measure $\mu \in M(X,T)$ is said to be  an \emph{equilibrium state} for $f$ with respect to $d$ if $$\over=F(\mu,d)+\int fd\mu.$$  
\end{df}

The set of  all equilibrium states for $f$ with respect to $d$ is  denoted by $M_f(T,X,d)$.
Recall that a measure  $\mu$ on $X$is said to be a  finite signed  measure  if  the  map $\mu: \mathcal{B}(X)\longrightarrow \R$ is countably additive, where $ \mathcal{B}(X)$ is the Borel sigma algebra of $X$.  We will see that  the notion of equilibrium state is  closely tied with the notion  of \emph{tangent functional} to  the convex function  ${\rm \overline{{mdim}}}_M(T,\cdot,d)$.

\begin{df}
Let $(X,T)$ be a TDS with a metric $\d$ staisfying ${\rm \overline{mdim}}_M(T,X,d)<\infty$, and let $\f$.  A \emph{tangent functional} to ${\rm \overline{{mdim}}}_M(T,\cdot,d)$ at $f$ with respect to $d$ is a  finite signed measure $\mu$ such that $${\rm \overline{{mdim}}}_M(T,f+g,d)-\over\geq \int gd\mu$$ for all $g\in C(X,\R)$. 
\end{df}
The set of  all tangent functionals to ${\rm \overline{{mdim}}}_M(T,\cdot,d)$ at $f$ with respect to $d$ is  denoted by $t_f(T,X,d)$.

\begin{rem}
(1)  The Riesz representation  theorem  states that  $C(X,\R)^{*}$ can be  identified with  the set of all  finite signed measures on $X$. Therefore, one can understand  the  tangent functional $\mu$ to ${\rm \overline{{mdim}}}_M(T,\cdot,d)$ at $f$ w.r.t. $d$  as  an element $L$ of $C(X,\R)^{*}$ such that $L(g)\leq{\rm \overline{{mdim}}}_M(T,f+g,d)-\over$ for all $g\in C(X,\R).$

(2) Consider the  linear space $\{c:c\in\R\}$. We have $$\gamma(c):=\int cd\mu \leq {\rm \overline{{mdim}}}_M(T,f+c,d)-\over$$ for all $c\in \mathbb{R}$.  Using Hahn-Banach theorem one can extend  $\gamma$ on $\R$ to  an element of $C(X,\R)^{*}$ dominated by the convex function $g\mapsto{\rm \overline{{mdim}}}_M(T,f+g,d)-\over$. Thus,  for each $f\in C(X,\R)$,  $t_f(T,X,d)$ is not  empty. 
\end{rem}

\begin{prop}
Let $(X,T)$ be a TDS with a metric $\d$  such that ${\rm \overline{{mdim}}}_M(T,X,d)<\infty$ and $\f$. Then 

(1) $M_f(T,X,d)$ is a  non-empty compact  convex subset of $M(X,T)$.

(2) \begin{align*}
M_f(T,X,d)&=t_f(T,X,d)\\
&=\cap_{n\geq 1} \overline{\{\mu \in M(X,T):  F(\mu,d)+\int fd\mu>\over-\frac{1}{n}\}}.
\end{align*}

(3) There exists a dense  subset $\mathcal{D}$ of $C(X,\R)$  such that for any $f\in \mathcal{D}$, $M_f(T,X,d)$  has  a unique equilibrium  state.
\end{prop}

\begin{proof}
(1) By Theorem 1.1, the set $M_f(T,X,d)$ is not nonempty. Let $\mu,\nu\in M_f(T,X,d)$  and $p\in [0,1]$. Then
\begin{align*}
\over&= p F(\mu,d)+(1-p) F(\nu)+p\int fd\mu +(1-p)\int fd\nu\\
&\leq F(p\mu+(1-p)\nu) +\int fdp\mu+(1-p)\nu~~\text{by Proposition \ref{prop 3.6}}\\
&\leq \over~~ \text{by Theorem \ref{thm 1.1}},
\end{align*} 
 which implies that  $M_f(T,X,d)$ is convex.
 
  Let  $\mu_n \to \mu$ with $\mu_n \in M_f(T,X,d)$.  As $F(\mu,d)$ is upper semi-continuous, we have 
  \begin{align*}
  \over= \limsup_{n\to \infty} \{F(\mu_n)+\int fd\mu_n\}\leq  F(\mu)+\int fd\mu,
  \end{align*}
which implies that $\mu\in M_f(T,X,d)$. Therefore,the set $ M_f(T,X,d)$ is closed and hence compact.

(2)  $M_f(T,X,d)=\cap_{n\geq 1} \overline{\{\mu \in M(X,T):  F(\mu,d)+\int fd\mu >\over-\frac{1}{n}\}}$ is clear. It remains to show$M_f(T,X,d)=t_f(T,X,d)$.

Let $\mu \in M_f(T,X,d)$. Then we have
\begin{align*}
{\rm \overline{{mdim}}}_M(T,f+g,d)-\over&\geq F(\mu,d)+\int f+gd\mu -(F(\mu,d)+\int fd\mu)\\
&=\int gd\mu.
\end{align*}
This  yields that  $\mu \in t_f(T,X,d)$ and hence $ M_f(T,X,d) \subset t_f(T,X,d)$.

On the other hand,  let $\mu \in t_f(T,X,d)$.  For any $g\in C(X,\R)$ with $g\geq 0$ and $\epsilon>0$,  we have
\begin{align*}
\int g+\epsilon d\mu&=-\int -(g+\epsilon)d\mu\\
&\geq -{\rm \overline{{mdim}}}_M(T,f-(g+\epsilon),d)+\over\\
&\geq  -{\rm \overline{{mdim}}}_M(T,f-\inf(g+\epsilon),d)+\over\\
&=\inf(g+\epsilon)>0.
\end{align*}
This shows $\mu$ is a non-negative measure on $X$.For $n\in \Z^{+}$,  we have
\begin{align*}
\int nd\mu \leq {\rm \overline{{mdim}}}_M(T,f+n,d)-\over=n,
\end{align*}
which tells us $\mu(X)\leq1$. If $n$ is a negative  integer number, one can  similarly obtain $\mu(X)\geq 1$. This shows $\mu \in M(X).$
For any $g\in C(X,\R)$, one has
$${\rm \overline{{mdim}}}_M(T,f+g,d)-\int f+gd\mu\geq \over-\int fd\mu.$$
The arbitrariness of $g$ implies that  for any $h\in C(X,\R)$,
$${\rm \overline{{mdim}}}_M(T,h,d)-\int hd\mu\geq \over-\int fd\mu.$$
So  $F(\mu,d)\geq\over-\int fd\mu$ by  Definition \ref{def 3.4}. The  Step 2  in the proof of Theorem \ref{thm 1.1}  asserts that every  $\mu \in M(X)$ attaining the supremum  must be$T$-invariant. This shows $\mu \in M_f(T,X,d)$ and hence $t_f(T,X,d)\subset M_f(T,X,d)$.

(3) The  well-known theorem \cite[p.450]{ds88} states that  for a convex function on a separable Banach space,  there exists  a  unique  tangent  functional at a  dense set  of points.  Applying  the theorem to ${\rm \overline{{mdim}}}_M(T,\cdot,d)$  and  $C(X,\R)$, together with  the relation $M_f(T,X,d)=t_f(T,X,d)$ in (2), this finishes the proof.
\end{proof}

\begin{cor}
Let $(X,T)$ be a  dynamical system with a metric $\d$  such that ${\rm \overline{{mdim}}}_M(T,X,d)<\infty$ and $\f$. Suppose that $\mu_s\in M(X,T)$, $s\in \R$, is an equilibrium state for $sf$ with $\int fd\mu_s\not=0$. Then 
\begin{align*}
{\rm \overline{{mdim}}}_M(T,sf,d)=0~~\text{ if and only if}~~s=-\frac{F(\mu_s,d)}{\int fd\mu_s}.
\end{align*}
\end{cor}
\begin{proof}
It is  straightforward to  use the  ${\rm \overline{{mdim}}}_M(T,sf,d)=F(\mu_s,d)+s\int fd\mu_s$.
\end{proof}

For  thermodynamic formalism of infinite entropy systems, the  present authors \cite{ycz22a} established Bowen's equations for upper metric mean dimension with potential to associate  different types of  metric mean dimension with potentials  within  dimension  theory of  dynamical systems.  As an application,  we will see that the  BS  metric mean dimension can be computed via equilibrium states.
\begin{cor}
Let $(X,T)$ be a dynamical system with a metric $\d$  such that ${\rm \overline{{mdim}}}_M(T,X,d)<\infty$ and $\f$ with $f>0$.  Then there exists  the unique $s_0\geq 0$ such that  for any  $\mu \in M_{-s_0f}(T,X,d)$, we have  $${\rm  \overline{BSmdim}}_{M}(f,X,d)=\frac{F(\mu_{s_0},d)}{\int fd\mu_{s_0}},$$
where  ${\rm  \overline{BSmdim}}_{M}(f,X,d)$  is  BS  metric mean dimension on $X$ with respect to $f$. See \cite[Definition 3.8]{ycz22a} for  its precise definition.
\end{cor}
\begin{proof}
By \cite[Theorem 1.3]{ycz22a}, we know that $s_0={\rm  \overline{BSmdim}}_{M}(f,X,d)\\\geq0$ is the  unique root of  the equation $${\rm \overline{{mdim}}}_M(T,-sf,d)=0.$$  For any  $\mu \in M_{-s_0f}(T,X,d)$, we have 
$${\rm \overline{{mdim}}}_M(T,-s_0f,d)=F(\mu_{s_0}, d)-s_0\int fd\mu_{s_0}.$$ Then${\rm  \overline{BSmdim}}_{M}(f,X,d)=\frac{F(\mu_{s_0},d)}{\int fd\mu_{s_0}}$ by the uniqueness of  the root of the equation.
\end{proof}
\section*{Acknowledgement} 

\noindent The work was supported by the
National Natural Science Foundation of China (Nos.12071222 and 11971236), China Postdoctoral Science Foundation (No.2016M591873),
and China Postdoctoral Science Special Foundation (No.2017T100384). The work was also funded by the Priority Academic Program Development of Jiangsu Higher Education Institutions.  We would like to express our gratitude to Tianyuan Mathematical Center in Southwest China(11826102), Sichuan University and Southwest Jiaotong University for their support and hospitality.


\end{document}